\documentclass[12pt,notitlepage]{amsart}
\usepackage{latexsym,amsfonts,amssymb,amsmath,amsthm}
\usepackage{graphicx}
\usepackage{color}
\usepackage{multirow}
\usepackage[normalem]{ulem}
\usepackage{datetime2}
\let\turc\c
\usepackage{etoolbox}
\robustify\turc 
\renewcommand{\c}{\mathfrak{c}}
\newcommand{\bpm}{\begin{pmatrix}}
\newcommand{\epm}{\end{pmatrix}}

\usepackage[inner=1.0in,outer=1.0in,bottom=1.0in, top=1.0in]{geometry}

\newcommand{\mz}{\ensuremath{\mathbb Z}}

\newcommand{\shortmod}{\ensuremath{\negthickspace \negthickspace \negthickspace \pmod}}

\DeclareMathOperator{\supp}{supp}

\newtheorem{theorem}{Theorem}[section]
\newtheorem{lemma}[theorem]{Lemma}

\newtheorem{proposition}[theorem]{Proposition}

\theoremstyle{definition}
\newtheorem{definition}[theorem]{Definition}
\theoremstyle{remark}
\newtheorem{remark}[theorem]{Remark}

\numberwithin{equation}{section}

\DeclareFontFamily{OML}{rsfs}{\skewchar\font'177}
\DeclareFontShape{OML}{rsfs}{m}{n}{ <5> <6> rsfs5 <7> <8> <9>
rsfs7 <10> <10.95> <12> <14.4> <17.28> <20.74> <24.88> rsfs10 }{}
\DeclareMathAlphabet{\mathfs}{OML}{rsfs}{m}{n}

\newcommand{\BQ}{{\mathbb{Q}}}

\newcommand{\BZ}{{\mathbb{Z}}}

\newcommand{\CA}{{\mathcal{A}}}
\newcommand{\CB}{{\mathcal{B}}}

\newcommand{\CD}{{\mathcal{D}}}

\newcommand{\CL}{{\mathcal{L}}}
\newcommand{\CM}{{\mathcal{M}}}

\newcommand{\CS}{{\mathcal{S}}}

\newcommand{\CW}{{\mathcal{W}}}

\newcommand{\Mod}[1]{\ (\mathrm{mod}\ #1)}

\newcommand{\bae}{\begin{equation}\begin{aligned}}
\newcommand{\eae}{\end{aligned}\end{equation}}

\def\beq{ \begin{equation} }
\def\eeq{ \end{equation} }

\def\square{\vcenter{\vbox{\hrule height .4pt
  \hbox{\vrule width .4pt height 5pt \kern 5pt
        \vrule width .4pt} \hrule height .4pt}}}

\begin{document}

\author{Bradford Garcia}
\email{garcia@math.tamu.edu}

\author{Matthew P. Young} 
\email{myoung@math.tamu.edu}
\address{Department of Mathematics\\
	  Texas A\&M University\\
	  College Station\\
	  TX 77843-3368\\
		U.S.A.}

		\thanks{This material is based upon work supported by the National Science Foundation under agreement No. DMS-2001306 and DMS-2302210 (M.Y.).  Any opinions, findings and conclusions or recommendations expressed in this material are those of the authors and do not necessarily reflect the views of the National Science Foundation.
 }

 \begin{abstract}
 We prove an asymptotic formula for the second moment of central values of Dirichlet $L$-functions restricted to a coset.
More specifically, consider a coset of the subgroup of characters modulo $d$ inside the full group of characters modulo $q$.  
Suppose that $\nu_p(d) \geq \nu_p(q)/2$ for all primes $p$ dividing $q$.  In this range, we obtain an asymptotic formula with a power-saving error term; curiously, there is a secondary main term of rough size $q^{1/2}$ here which is not predicted by the integral moments conjecture of Conrey, Farmer, Keating, Rubinstein, and Snaith.  The lower-order main term does not appear in the second moment of the Riemann zeta function, so this feature is not anticipated from the analogous archimedean moment problem. 

We also obtain an asymptotic result for smaller $d$, with $\nu_p(q)/3 \leq \nu_p(d) \leq \nu_p(q)/2$, with a power-saving error term for $d$ larger than $q^{2/5}$.  In this more difficult range, the secondary main term somewhat changes its form and may have size roughly $d$, which is only slightly smaller than the diagonal main term.  
 \end{abstract}

\title{Asymptotic second moment of Dirichlet $L$-functions along a thin coset}
\maketitle

\section{Introduction}
The study of moments of families of $L$-functions has a long history.  One strand of research concerns the estimation of moments in order to secure strong subconvexity bounds.  Another direction is to consider the structural properties of the moments, especially in their connections with random matrix theory.  One of the long-standing guiding principles concerns analogies between different families, particularly the consideration of the Riemann zeta function in the $t$-aspect on the one side, and the family of Dirichlet $L$-functions in the $q$-aspect on the other side.  

We begin by discussing some moment problems for the zeta function.
This is a vast subject, and we only briefly touch on a few results pertinent to our narrative.
In \cite{HB2}, Heath-Brown studied the twelfth moment of the Riemann zeta function, finding that
\begin{equation}\label{HB}
    \int_T^{2T} |\zeta(1/2+it)|^{12} dt \ll_\varepsilon T^{2+\varepsilon},
\end{equation}
a result which easily recovers the Weyl bound while also proving in a strong quantitative form that $|\zeta(1/2+it)|$ cannot be too large too often.  Heath-Brown's technique for proving \eqref{HB} is based on leveraging information from short second moments of the zeta function.  A simple modification of \cite[Lemma 1]{HB2} implies
\begin{equation}
\label{eq:HBshortsecondmoment}
 \int_T^{T+T^{1/3}} |\zeta(1/2+it)|^2 dt \ll T^{1/3+\varepsilon},
\end{equation}
for $T \geq 1$.  See also \cite[Section 7.4]{Titchmarsh} and \cite[Chapter 7]{Ivic} for more discussion and other related results.
For example, \cite[Theorem 7.4]{Titchmarsh} gives that
\begin{equation}
\label{eq:zetasecondmomentwithErrorTerm}
 \int_0^{T} |\zeta(1/2+it)|^2 dt = T \log T + (2 \gamma_0 - 1 - \log(2 \pi))T + O(T^{1/2+\varepsilon}),
\end{equation}
and it is remarked that this error term can be reduced to $O(T^{5/12+\varepsilon})$.  
Note that \eqref{eq:zetasecondmomentwithErrorTerm} leads to an asymptotic formula for the second moment in a short interval of length at least $T^{1/2+\varepsilon}$ (or $T^{5/12 + \varepsilon}$ with the alluded-to improved method), simply by writing $\int_T^{T+\Delta} = \int_0^{T+\Delta}  - \int_0^{T}$.

Next we discuss some prior works on $q$-analogs of these results.  
Nunes \cite{N} proved an analog to \eqref{HB} for smooth square-free moduli $q$.  In a complementary direction, 
Milićević and White \cite{MW} proved a variant of \eqref{HB} in the depth aspect, i.e., fixing a prime $p$ and letting $q=p^j$ with $j \rightarrow \infty$.  Both of these works consider upper bounds on ``short'' second moments, where in this context ``short'' refers to a coset.  
Petrow and Young \cite{PetrowYoung} obtained an upper bound on the fourth moment of Dirichlet $L$-functions along a coset of size $\gg q^{2/3+\varepsilon}$.
For the full family second moment, Heath-Brown \cite{HBDirichletAsymptotic}
proved
\begin{equation}
\label{eq:HBsecondmoment}
 \sum_{\chi \shortmod{q}} |L(1/2, \chi)|^2
 = \frac{\phi(q)}{q} \sum_{d | q} \mu(q/d) T(d), 
\end{equation}
where with $\gamma_0 = 0.577\dots$ representing Euler's constant, we have
\begin{equation*}
 T(d) = d( \log \tfrac{d}{8 \pi} + \gamma_0) + 2 \zeta(1/2)^2 d^{1/2} + \sum_{n=0}^{2N-1} c_n d^{-n/2} + O(d^{-N}).
\end{equation*}
One remark that is in order here is that Heath-Brown considers a sum over all Dirichlet characters, and not just the primitive ones.  A second comment is that there is no obvious way to extract from \eqref{eq:HBsecondmoment} a formula for the second moment restricted to a coset, in contrast to the $t$-integral analog for the zeta function.  Finally, we remark that for composite values of $q$, there are a variety of main terms that may be intermediate in size between $q^{1/2}$ and $q$.  On the other hand, the five authors \cite{CFKRS} have conjectured an asymptotic formula for the second moment in this family, but only after restriction to the family of solely the primitive characters modulo $q$.  Conrey \cite{Conrey} verified the five author conjecture for the second moment in this family.  In the same paper, Conrey also derived a beautiful reciprocity formula for the twisted second moment with prime modulus.

To state our main results, we need some notation.  
Suppose $\psi$ is a Dirichlet character modulo $q$ and $d$ is a positive integer with $d|q$ and $q|d^2$.  
 It is easy to see that, as a function of $x$, $\psi(1+dx)$ is an \emph{additive} character with period $q/d$.  Hence
there exists an integer $a_{\psi} \pmod{q/d}$ such that
  $\psi(1+dx) = e(\frac{a_\psi dx}{q})$, where $e(x)=e^{2\pi ix}$. 
  The primitivity of $\psi$ modulo $q$ means that $(a_{\psi}, q/d) = 1$ (for which see Lemma \ref{postnikov} below).
  We assume
\begin{equation}\label{a range}
    0<|a_\psi|<\frac{q}{2d}.
\end{equation}

 Let $\nu_p(\cdot)$ denote the $p$-adic valuation.  For $a$ and $b$ be positive integers, we will write ``$a\prec b$" to mean that $a$ and $b$ share all of the same prime factors and, for any prime $p|b$, $\nu_p(a)<\nu_p(b)$. Similarly, we will write ``$a\preceq b$" to mean that $a$ and $b$ share all of the same prime factors and, for any prime $p|b$, $\nu_p(a)\leq\nu_p(b)$. 

\begin{theorem}\label{main1}
Let $\psi$ be even of conductor $q$, and suppose 
$d \prec q \preceq d^2$.
Then
\begin{equation}\label{main1.1}
    \sum_{\substack{\chi \Mod{d}\\ \chi \text{ even}}} |L(1/2,\chi \cdot \psi)|^2 = \CD+\CA +
    O(q^{-1/8+ \varepsilon} d),
\end{equation}
where with $\theta(q) = \sum_{p|q} \frac{\log{p}}{p-1}$, and $\sigma_\alpha(n) = \sum_{k|n} k^\alpha$, we define
\begin{equation}\label{main1.2}
    \CD = \frac{\varphi(d)}{2}\frac{\varphi(q)}{q}(\log{q} + 2 \gamma_0 + \tfrac{\Gamma'}{\Gamma}(1/4) - \log{\pi} + 2 \theta(q)),
\end{equation}
and
\begin{equation}
    \CA = \frac{\varphi(d)}{d} \sqrt{q} \; 
    \frac{\sigma_0(|a_\psi|)}{\sqrt{|a_\psi|}}.
\end{equation}
\end{theorem}
Remarks.  The term $\CD$ is the diagonal main term expected from \cite{CFKRS}, and which is analogous to the main term in \eqref{eq:zetasecondmomentwithErrorTerm}, but $\CA$ is not predicted by the recipe of \cite{CFKRS}.  The existence of $\CA$ contrasts with the absence of a main term of size $T^{1/2}$ in \eqref{eq:zetasecondmomentwithErrorTerm}.
It also indicates that care must be taken when using the recipe of \cite{CFKRS} for sub-families.

The condition that $d \prec q$ implies that $\chi \cdot \psi$ has conductor $q$ for all $\chi \pmod{d}$, and so this lower-order main term $\CA$ is not due to characters of smaller modulus (as in \eqref{eq:HBsecondmoment}).
If $a_{\psi} = 1$, say, then $\CA = \frac{\varphi(d)}{d} \sqrt{q}$, which is essentially its maximal size.
However, note that $\CA \ll \sqrt{q}$ which is smaller than $\CD$, since $\sqrt{q} \leq d$, and $\CD \gg d (\log q)^{1-\varepsilon}$.

As a very rough heuristic, we can indicate how $\CA$ arises. 
Applying an approximate functional equation and the orthogonality formula, we encounter a sum of the form
\begin{equation*}
\varphi(d) \sum_{m \equiv n \shortmod{d}} \frac{\psi(m) \overline{\psi(n)}}{\sqrt{mn}}.
\end{equation*}
With $m= n + d l$, we have $\psi(m) \overline{\psi}(n) = \psi(1 + d l \overline{n}) = e_{q/d}(a_{\psi} l \overline{n})$.  
Consider the term $n=1$.  If $a_{\psi} l$ is a bit smaller than $q/d$, then the exponential is approximately $1$.  Once $|a_{\psi}| l$ is larger than $q/d$, then the exponential
 has cancellation and should not contribute to a main term.
This thought process leads to
\begin{equation}
\label{eq:handwavy}
\varphi(d) \sum_{l \geq 1}  \frac{e_{q/d}(a_{\psi} l)}{\sqrt{d l}}
 \approx \frac{\varphi(d)}{\sqrt{d}} \sum_{l < \frac{q}{d |a_{\psi}|}} \frac{1}{\sqrt{l}}
 \approx \frac{\varphi(d) q^{1/2}}{d \sqrt{|a_{\psi}|}}.
\end{equation}
One can apply the same sort of reasoning for each $n$ dividing $a_{\psi}$, giving an expression of the same form, in line with the presence of $\sigma_0(|a_{\psi}|)$ in $\CA$.  Although this heuristic is far from rigorous, it is true that $\CA$ emerges in the analysis from the most unbalanced range of summation where $n$ is very small and $m=n+dl$ is very large.

  For the next theorem, we will have that $d|\frac{q}{d}$, $q | d^3$, and $(q,3) = 1$,
  in which case we define $a_{\psi} \pmod{q/d}$ by the condition that $\psi(1+dx) = e_{q/d}(a_{\psi} ( x - \overline{2}dx^2))$ for all $x \in \mz$.  For more details on this definition, see Section \ref{section:Postnikov}.
  Note this definition reduces to the earlier one when $d \equiv 0 \pmod{q/d}$ since then the quadratic term may be discarded.
In addition, let $b_\psi$ be the reduction of $a_\psi \Mod{d}$ such that
$0<|b_\psi|<\frac{d}{2}$.

\begin{theorem}\label{main2}
Suppose $\psi$ is even of conductor $q$, with $(q,3) = 1$ and suppose 
$d^2 \preceq q \preceq d^3$.
Then
\begin{equation}\label{main2.1}
    \sum_{\substack{\chi \Mod{d}\\ \chi \text{ even}}} |L(1/2,\chi \cdot \psi)|^2 = \CD+\CA' + 
  O(d^{-1/4} q^{1/2+\varepsilon}),
\end{equation}
where $\CD$ is as defined in \eqref{main1.2} and with $a_{\psi} \overline{a_{\psi}} \equiv 1 \pmod{q}$ we set
\begin{equation}\label{main2.2}
    \CA' = 
\Big(\frac{2 a_{\psi}}{q} \Big) \cdot \varphi(d)  \frac{\sigma_0(|b_\psi|)}{\sqrt{|b_\psi|}}
\times   
    \begin{cases}   \cos\Big(2\pi\frac{ \overline{2a_\psi}(a_\psi-b_\psi)^2}{q}\Big) , & q\equiv1\Mod{4}
    \\
    \sin\Big(2\pi\frac{\overline{2a_\psi}(a_\psi-b_\psi)^2}{q}\Big), & q\equiv3\Mod{4}. \end{cases}
\end{equation}
\end{theorem}
Remarks.  First observe that the diagonal term $\CD$ is larger than the error term provided $d \gg q^{2/5+\varepsilon}$, which goes below the $\sqrt{q}$ threshold.   Secondly, although $\CA'$ can be negative, $|\CA'|$ is smaller than $\CD$, since $\CD \gg \varphi(d) (\log q)^{1-\varepsilon}$.  

  The presence of trigonometric functions at complicated arguments in $\CA'$ is a new feature compared to $\CA$ in Theorem \ref{main1}, and worthy of further discussion.  One simple observation is that if $|a_{\psi}| < d/2$, then $b_{\psi} = a_{\psi}$, and it simplifies as $\cos(0) = 1$ or $\sin(0) = 0$.  
On the other hand, there exist situations with $q \equiv 3 \pmod{4}$ where $\CA'$ cannot be discarded.  For example, let $q = p^{239}$ and $d = p^{116}$, and suppose $p \rightarrow \infty$.  Here $d > q^{2/5}$, so the error term is smaller than the diagonal term $\CD$.  Now suppose $\psi$ is a character with $a_{\psi} = 1 + 2p^{116}$, which is a valid range since we need $|a_{\psi}| < \frac12 p^{123}$.  Then $b_{\psi} =1$, and so $a_{\psi} - b_{\psi} =  2p^{116}$.  Under these conditions, we have
\begin{equation*}
  \overline{2 a_{\psi}} (a_{\psi} - b_{\psi})^2
=  \overline{2(1+ 2p^{116})} 4 p^{232}
\equiv 2 p^{232} \pmod{p^{239}}.
\end{equation*}
Therefore,
\begin{equation*}
    \CA' = 
 \varphi(d) 
\times   
    \begin{cases}   \cos(4\pi p^{-7}) , & p\equiv1\Mod{4}
    \\
    \sin(4\pi p^{-7}), & p\equiv3\Mod{4}. \end{cases}
\end{equation*}
In the case $p \equiv 3 \pmod{4}$, then $|\CA'| \approx d p^{-7} = p^{109}$.   This is larger than the error term which has size
$d^{-1/4} q^{1/2} = p^{-116/4 + 119.5} \leq p^{91}$.
One can construct other examples exhibiting other types of behavior for $\CA'$.  Compared to the discussion around \eqref{eq:handwavy}, it seems harder to heuristically see the shape of $\CA'$.  However, we stress that in the proof it arises in the same way as $\CA$, with the main differences coming from requiring a quadratic approximation for $\psi(1+d l \overline{n})$.  The relevant sums can be evaluated in closed form using quadratic Gauss sums. 

  Theorem \ref{main2} can be extended with little effort to $q$ with $3|q$ (with a slightly more restrictive assumption $q\preceq \frac{1}{3}d^3)$.  One could also attempt to find a common/hybrid generalization of Theorems \ref{main1} and \ref{main2} by 
  only requiring $d \prec q$ and $q \preceq d^3$.  To do so, one could factor $d = d_1 d_2$ where $d_1$ ($d_2$, resp.) consists of the part of $d$ where $\nu_p(d_1) \geq \nu_p(q)/2$ ($\nu_p(d_2) < \nu_p(q)/2$, resp.), and use the ideas of proof of Theorem \ref{main1} for $d_1$ (Theorem \ref{main2} for $d_2$, resp.).  Our separation of Theorem \ref{main1} and \ref{main2} is intended to simplify the exposition.
  
 It is well-known that for many families of $L$-functions it may be easier to obtain an upper bound in place of an asymptotic formula. 
We have the following upper bound, which is sharp for $d \gg q^{1/3+\varepsilon}$.
\begin{theorem}\label{main3}
Let $\psi$ have conductor $q$, and let $d|q$. Then
\begin{equation}
\label{eq:main3}
    \sum_{\chi \Mod{d}} |L(1/2,\chi \cdot \psi)|^2 \ll q^\varepsilon (d + d^{-1/2} q^{1/2}).
\end{equation}
\end{theorem}
Special cases of this result appear in both \cite{N, MW}.  One nice consequence of Theorem \ref{main3} is that if $q$ has a factor $d$ with $d = q^{1/3 + o(1)}$, then this second moment is strong enough to recover the Weyl bound $|L(1/2, \psi)| \ll q^{1/6+\varepsilon}$.  For moduli $q$ with $d|q$, Heath-Brown \cite[Theorem 2]{HB1} proved a bound 
for an individual $L$-function which essentially matches \eqref{eq:main3}.
Indeed, our proof of Theorem \ref{main3} relies on Heath-Brown's work, and in retrospect the method of Heath-Brown is implicitly bounding a second moment along a coset.  Of course, the second moment bound in \eqref{eq:main3} implies an individual bound, but it contains more information.

\section{Preliminaries}

  In this section, we will lay the groundwork and develop the tools necessary to prove the main theorems.

\subsection{Various Bounds \& Evaluations}

First we define the sum studied by Heath-Brown in \cite{HB1} and cite his associated bound.
\begin{definition}
Let $\chi$ have conductor $q$, and let $h$ and $n$ be integers. Denote
\begin{equation} \label{HB sum}
    S(q;\chi,h,n) = \displaystyle\sum_{m=0}^{q-1} \chi(m+h) \overline\chi(m) e(mn/q).
\end{equation}
\end{definition}

\begin{lemma}[Heath-Brown, \cite{HB1}, Lemma 9]\label{HB result}
Suppose that $q$ is odd, $q_0 | q$, and $\varepsilon > 0$. Then
\begin{equation} \label{HB1}
    \sum_{1 \leq |h| \leq A} \sum_{1 \leq |n| \leq B} |S(q;\chi,hq_0,n)| \ll q^{1/2+\varepsilon} (ABq_0^{-1/2} + (qAq_0)^{1/4}),
\end{equation}
and
\begin{equation} \label{HB2}
    \sum_{1 \leq |h| \leq A} |S(q;\chi,hq_0,0)| \ll q_0A q^{\varepsilon}.
\end{equation}
\end{lemma}

\begin{remark}
Technically, \cite[Lemma 9]{HB1} gives a bound of $|S(q;\chi,4hq_0,n)|$, without the restriction that $(q,2) = 1$. However, it can be seen by reading through his proof that the result holds for $|S(q;\chi,hq_0,n)|$ (without the $4$) with the condition that $q$ is odd.  Moreover, Heath-Brown claims the bound for the sum with $1 \leq h \leq A$ (and $1 \leq n \leq B$ in \eqref{HB1}), but simple symmetry arguments extend the result as claimed above.
\end{remark}

  We next state the standard evaluation of a quadratic exponential sum.
\begin{lemma}\label{quad sum}
Let $r$ be a positive odd integer. Let $A,B$ be integers such that $(B,r)=1$. Then
\begin{equation*}
    \sum_{u \Mod{r}} e_{r}\left(Au + Bu^2\right) = e_r(-\overline{4B}A^2 ) \left(\frac{B}{r}\right) \varepsilon_{r} \sqrt{r},
\end{equation*}
where $e_q(x) = e(\frac{x}{q})$, the bar notation indicates the multiplicative inverse modulo $r$, $\left(\frac{B}{r}\right)$ is the Jacobi symbol, and $\varepsilon_r = \begin{cases} 1, & r\equiv 1\Mod{4}\\ i, & r\equiv 3\Mod{4}. \end{cases}$
\end{lemma}

\begin{proof}
This follows by completing the square and applying (3.38) from \cite{IK}.
\end{proof}

  Another simple lemma to get us warmed up is the following integral evaluation.
\begin{lemma}\label{Mellin}
Let $k$ be a non-zero integer and suppose $-\frac{1}{2}<\Re(s)<\frac{1}{2}$. Then
\begin{equation*}
    \int_0^\infty  x^{-s}  \cos(2 \pi k x) \frac{dx}{\sqrt{x}} =  \frac{\Gamma(1/2-s)}{(2\pi|k|)^{1/2-s}} 
 \cos(\tfrac{\pi}{2}(\tfrac12 - s)) .
\end{equation*}
\end{lemma}

\begin{proof}
This follows from
\cite[17.43.3]{GR} after changing variables.
\end{proof}

\subsection{Postnikov}
\label{section:Postnikov}
We will derive a variant of the Postnikov formula that holds for composite moduli, rather than only prime powers.
Recall the notation $a|b^\infty$, which means that $a| b^A$ for some $A > 0$.

\begin{definition}\label{d-adic log}
For positive odd integers $d$ and $q$ such that $d|q$ and $q|d^\infty$, define the formal power series in the indeterminate $x$, by
\begin{equation} \label{L_q}
    \CL_q(1+dx) = \sum_{k=1}^\infty (-1)^{k+1}\;\frac{d^k}{k}\;x^k \in \BQ[[x]].
\end{equation}
\end{definition}
The conditions $d|q$ and $q|d^\infty$ imply that $d$ and $q$ share all of the same prime factors.  
We first show some divisibility properties for the coefficients of this formal power series.  Let
\begin{equation}\label{R_q}
    R_q=\{x\in\BQ \text{ such that } \nu_p(x)\geq0 \text{ for all } p|q\}.
\end{equation}
Note that $R_q$ is a sub-ring of $\BQ$.
Reduction modulo $q$ gives rise to the ring homomorphism $\varphi : R_q \to \BZ/q\BZ$ given by  
\begin{equation}\label{reduction}
    \varphi(a/b)\equiv a\overline{b} \pmod{q}, \quad \text{where} \quad  b \overline{b} \equiv 1 \pmod{q}.
\end{equation}
\begin{lemma}\label{p-adic lem}
Let $d \preceq q$ with $q$ odd. For any $p|q$ and $k \geq 1$ we have
\begin{equation}\label{p-adic arg}
    \nu_p(k)\leq\nu_p(d^{k-1}).
\end{equation}
More generally, for any $A \geq 0$ there exists $N$ such that $\nu_p(k)\leq\nu_p(d^{k-A})$ for all $k\geq N$.
\end{lemma}

\begin{proof}
We have $\nu_p(k) \leq \frac{\log(k)}{\log(p)} \leq k-1 \leq \nu_p(d^{k-1})$, where the last inequality follows from the fact that $d$ and $q$ share all of the same prime factors, so $\nu_p(d) \geq 1$. Now, \eqref{p-adic arg} follows. In the more general case, $\nu_p(d^{k-A})\geq k-A$ and $\nu_p(k)=O(\log(k))$, so there exists a large enough choice of $N$ such that $\nu_p(k) \leq \nu_p(d^{k-A})$ for all $k\geq N$.
\end{proof}

\begin{remark}\label{p-adic case}
While Lemma \ref{p-adic lem} only guarantees the existence of a positive integer $N$, the method of proof shows that a constructive candidate is the minimal positive integer $M$ such that $k-A\geq \log(k)$ for all $k\geq M$. This minimal $M$ can be easily found for any particular $A$, and an example of future relevance is that $M=5$ when $A=3$.
\end{remark}

\begin{lemma}\label{coefficients}
We have $\CL_q(1+dx) \in R_q[[x]]$.  In fact, all the coefficients of $\CL_q(1+dx)$ are multiples of $d$.
\end{lemma}

\begin{proof}
With $c_k=(-1)^{k+1}d^k/k$, we will show $\nu_p(c_k) \geq \nu_p(d)$ for all $p|q$. Using \eqref{p-adic arg} implies that $0 \leq \nu_p(\frac{d^{k-1}}{k})$, so $\nu_p(c_k) = \nu_p(d) + \nu_p(\frac{d^{k-1}}{k}) \geq \nu_p(d)$, for any $k$.
\end{proof}

  Since $\CL_q(1+dx)$ lives in $R_q[[x]] \subseteq \BQ[[x]]$, given $\varphi$ from \eqref{reduction}, there is an induced ring homomorphism $\overline\varphi : R_q[[x]] \to \left(\BZ/q\BZ\right)[[x]]$ which maps 
\begin{equation} \label{induced map}
    \overline\varphi(\CL_q(1+dx))=\sum_{k=1}^\infty \varphi(c_k) x^k.
\end{equation}
By abuse of notation, we  view $\CL_q(1+dx)$ as an element of $\left(\BZ/q\BZ\right)[[x]]$ via \eqref{induced map}.  

\begin{lemma} \label{finiteness}
We have $\CL_q(1+dx)\in(\BZ/q\BZ)[x]$.  Moreover, $\mathcal{L}_q(1+dx) \in d (\mz/q\mz)[x]$.  
\end{lemma}

\begin{proof}
The content of this lemma is that for sufficiently large $k$, $\nu_p(c_k) \geq \nu_p(q)$ for $c_k = d^k/k$.
Say $q | d^A$.
By Lemma \ref{p-adic lem}, there exists a positive integer $N$ such that $\nu_{p}(k)\leq\nu_{p}(d^{k-A})$ for all $k\geq N$. Then $\nu_p(c_k) = \nu_p(d^A) + \nu_p(\frac{d^{k-A}}{k}) \geq \nu_p(q)$ for $k \geq N$.  Here $N$ may depend on $p$, but by choosing the maximal of all these $N$'s gives a uniform value.
\end{proof}

  This opens the door to discussing various properties of $\CL_q(1+dx)$ modulo $q$, such as the following periodicity and additivity properties. These lemmas will require two indeterminates, so we will embed $(\BZ/q\BZ)[x]$ into $(\BZ/q\BZ)[x,y]$ in the obvious way.
\begin{lemma} \label{periodicity}
We have $\CL_q(1+dx)=\CL_q\left(1+d(x+\frac{q}{d}y)\right)$ in $(\BZ/q\BZ)[x,y]$.
\end{lemma}
\begin{proof}
We have $\CL_q(1 + d(x+\frac{q}{d} y)) = \sum_{k} c_k (x + \frac{q}{d} y)^k \equiv \sum_k c_k x^k \pmod{q}$, using $c_k \equiv 0 \pmod{d}$ from Lemma \ref{finiteness}. 
\end{proof}

\begin{lemma} \label{additivity}
We have $\CL_q((1+dx)(1+dy)) = \CL_q(1+dx) + \CL_q(1+dy)$ in $(\BZ/q\BZ)[x,y]$.
\end{lemma}

  \begin{proof}
We give a brief sketch here, and refer to \cite[pp.79-80]{K} for more details. 
  For the real logarithm, we have $\log((1+dx)(1+dy)) = \log(1+dx) + \log(1+dy)$, for $dx > -1$, $dy > -1$. Hence, the power series expansions of these two expressions are equal wherever they both converge, meaning that all of their corresponding coefficients are equal. Thus, since $\CL_q(1+dx)$ matches the power series expansion of the real logarithm (reduced modulo $q$ via \eqref{induced map}), this property also holds for $\CL_q(1+dx)$.
\end{proof}

\begin{lemma}\label{modularity}
Let $q$ and $d$ be positive odd integers with $d|q$.
\begin{enumerate}
\item If $q|d^\infty$, then $\CL_q(1+dx) \equiv dx \Mod{(q,d^2)}$.
\item If $(q,3)=1$ and $q|d^3$, then $\CL_q(1+dx) \equiv dx-\overline{2}(dx)^2 \Mod{q}$.
\end{enumerate}
\end{lemma}

\begin{proof}
Expanding the power series, we have
\begin{equation*}
    \CL_q(1+dx) = dx-\tfrac{1}{2}(dx)^2+\tfrac{1}{3}(dx)^3-d^3\Big(\underbrace{\tfrac{1}{4}dx^4-\tfrac{1}{5}d^2x^5+\ldots}_{\in (\BZ/q\BZ)[x]}\Big).
\end{equation*}
The claim that the tail of this series is still a polynomial with coefficients in $\BZ/q\BZ$ after factoring out $d^3$ follows from Lemma \ref{p-adic lem}, or more specifically the observations in Remark \ref{p-adic case}. Elaborating on the details, since $k-3\geq \log(k)$ for all $k\geq5$, we may choose $N=5$ from Lemma \ref{p-adic lem}. However, since $q$ is odd, $\frac{1}{4}=\overline{2}^2$ in $\BZ/q\BZ$, so we could also pick $N=4$. 
Now the result follows in each case.
\end{proof}

  We are now ready for our version of the Postnikov formula.
\begin{lemma}[Postnikov formula]
\label{postnikov}
Let $q$ and $d$ be odd with $d|q$ and $q|d^\infty$. 
There exists a unique group homomorphism 
$a: \widehat{(\mz/q \mz)^{*}} \rightarrow \mz/(q/d) \mz$, $\psi \mapsto a_{\psi}$,
such that a Postnikov-type formula holds: for each Dirichlet character $\psi$ modulo $q$ and $x\in\BZ$ we have 
\begin{equation}
\label{eq:postnikov}
    \psi(1+dx) = e_q(a_\psi \CL_q(1+dx)).
\end{equation}
Finally, if $\psi$ is primitive modulo $q$ then $(a_{\psi}, q/d) = 1$.
\end{lemma}

\begin{proof}
Consider the reduction modulo $d$ map
$(\BZ/q\BZ)^* \to (\BZ/d\BZ)^*$,
and denote its kernel by $K$. Since $d$ and $q$ share their prime factors, $K=\left\{1+dx : x\Mod{\frac{q}{d}}\right\}$, so $|K|=\frac{q}{d}$. Consider the map $f:K\to S^1$ defined by
\begin{equation*}
    f(1+dx)=e_q(\CL_q(1+dx)).
\end{equation*}
The function $f$ is well-defined by Lemmas \ref{finiteness} and \ref{periodicity}. Furthermore, $f$ is a group homomorphism by Lemma \ref{additivity}. Thus, $f$ is a character on $K$, and we claim that $f$ has order $\frac{q}{d}$ in $\widehat{K}$. Indeed, if $p$ is a prime such that $p|\frac{q}{d}$, then we have $f(1+dx)^{q/dp}=e_q\left(\frac{q}{dp}\CL_q(1+dx)\right)=e_{dp}\left(\CL_q(1+dx)\right)=e_{dp}(dx)=e(\frac{x}{p})$, by Lemma \ref{modularity} since $dp|(q,d^2)$. Hence, $\widehat{K}$ is cyclic and generated by $f$. Therefore, every element of $\widehat{K}$ is of the form $f^a$ for some $a \Mod{\frac{q}{d}}$. In particular, $\psi$ is a character on $(\BZ/q\BZ)^*$, so restricting $\psi$ to $K$ makes it an element of $\widehat{K}$. Thus, there exists a unique $a_\psi \Mod{\frac{q}{d}}$ such that $\psi|_K = f^{a_\psi}$, which is equivalent to the Postnikov formula.

For the final statement, suppose that $p$ is a prime with $p|\frac{q}{d}$ and $p|a_{\psi}$; we will show that $\psi$ is periodic modulo $q/p$.  
Since $d$ and $q$ share all the same prime factors, then $p^2|q$, and hence $q | (q/p)^{\infty}$.
Applying \eqref{eq:postnikov} with $d$ replaced by $\frac{q}{p}$, we obtain that $\psi(1+\frac{q}{p} y ) = e_q(a_{\psi} \frac{q}{p} y)$.  Hence if $p|a_{\psi}$ then $\psi$ is periodic modulo $q/p$.
\end{proof}

\subsection{Miscellaneous Lemmas}

  In the course of this paper, we encounter sums of the form
\begin{equation}\label{post sum}
    \CS_{q,d}(\psi,k) := \sum_{u \Mod{\frac{q}{d}}} \psi(1+du) e_q(dku). 
\end{equation}
We may evaluate this explicitly in many cases.  The conditions relevant for Theorem \ref{main2} are contained in the following.
\begin{lemma}\label{sum eval}
Let $(q,3) =1$, $\psi$ have conductor $q$, and suppose $d^2|q$ and $q|d^3$. Also, let $k \in \mz$. 
If $k \not \equiv -a_{\psi} \Mod{d}$ then $\CS_{q,d}(\psi,k) = 0$.  If $k \equiv - a_{\psi} \Mod{d}$ then
\begin{equation*}
    \CS_{q,d}(\psi,k) = 
\varepsilon_{q } \sqrt{q}   
    e_q(\overline{2a_\psi}(k+a_\psi)^2) \left(\frac{- 2a_\psi}{q/d^2}\right).  
\end{equation*}
\end{lemma}

\begin{proof}
By Lemma \ref{postnikov}, we have 
\begin{align*}
    \CS_{q,d}(\psi,k) &= \sum_{u \Mod{\frac{q}{d}}} e_q(dku+a_\psi \CL_q(1+du)).
\end{align*}
By Lemma \ref{modularity}(2), this simplifies as
\begin{equation*}
   \sum_{u \Mod{\frac{q}{d}}} e_q\left(dku+a_\psi\left(du - \overline{2}(du)^2\right)\right)
    = \sum_{u \Mod{\frac{q}{d}}} e_{q/d}\left((k+a_\psi)u - 2 a_\psi du^2\right).
\end{equation*}
Since $q/d^2$ is an integer by hypothesis, we may change variables $u \rightarrow u + q/d^2$ to 
see that the sum vanishes unless $k \equiv -a_\psi \Mod{d}$. We continue under this assumption, which means that the summand is actually periodic modulo $q/d^2$, and hence
\begin{equation*}
    \CS_{q,d}(\psi,k) 
    = d\sum_{u \Mod{\frac{q}{d^2}}} e_{q/d^2}\Big(\Big(\frac{k+a_\psi}{d}\Big)u - 2 a_\psi u^2\Big).
\end{equation*}
Lemma \ref{quad sum}, along with some simplification, concludes the proof.
\end{proof}

We can also easily calculate $\CS_{q,d}$ under the conditions relevant for Theorem \ref{main1}, as follows.
\begin{lemma}\label{sum eval2}
Suppose $\psi$ has conductor $q$, $d|q$, and $q|d^2$. Let $k \in \mz$. 
Then
\begin{equation*}
\CS_{q,d}(\psi,k)
=
\begin{cases}
q/d, \qquad & k \equiv - a_{\psi} \pmod{q/d} \\
0, \qquad & k \not \equiv - a_{\psi} \pmod{q/d}.
\end{cases}
\end{equation*}
\end{lemma}
The proof of Lemma \ref{sum eval2} is similar to, but much easier than, Lemma \ref{sum eval}, since in this case we may discard the quadratic terms in $\CL_q(1+du)$.  We omit the details.

\begin{lemma}[Iwaniec-Kowalski, \cite{IK}, Theorem 5.3]\label{afe}
Let $\chi$ be a primitive even character modulo $q$. Suppose $G(s)$
is holomorphic and bounded in the strip $-4 < Re(s) < 4$, even, and normalized by
$G(0) = 1$. Then
\begin{equation*}
    L(1/2,\chi)L(1/2,\overline\chi) = 2 \sum_{m,n \geq 1} \frac{\chi(m)\overline\chi(n)}{\sqrt{mn}} V\left(\frac{mn}{q}\right),
\end{equation*}
where $V(x)$ is a smooth function defined by
\begin{equation}\label{V}
    V(x) = \frac{1}{2\pi i} \int_{(1)} \frac{G(s)}{s} \frac{\gamma(1/2+s)^2}{\gamma(1/2)^2} x^{-s} ds
\end{equation}
and
\begin{equation}\label{gamma}
    \gamma(s) = \pi^{-s/2} \Gamma\left(\frac{s}{2}\right).
\end{equation}
\end{lemma}

  This next lemma encompasses the opening moves to prove Theorems \ref{main1} and  \ref{main2}.
\begin{lemma}\label{gambit}
Let $\psi$ be even and have conductor $q$ and suppose $d\prec q$. Then
\begin{equation}
\label{eq:Mdef}
 \mathcal{M}:=   \sum_{\substack{\chi \Mod{d}\\ \chi \text{ even}}} |L(1/2,\chi \cdot \psi)|^2 = \varphi(d) \sum_\pm \sum_{\substack{m \equiv \pm n \Mod{d}\\(mn,q)=1}} \frac{\psi(m)\overline\psi(n)}{\sqrt{mn}} V\left(\frac{mn}{q}\right).
\end{equation}
\end{lemma}

\begin{proof}
Using $\overline{L(1/2,\chi)} = L(1/2,\overline{\chi})$, then
\begin{equation*}
    \CM = \sum_{\substack{\chi \Mod{d}\\ \chi \text{ even}}} L(1/2,\chi \cdot \psi) L(1/2,\overline{\chi \cdot \psi}).
\end{equation*}
The character $\chi\cdot\psi$ is primitive modulo $q$ because $\psi$ is primitive modulo $q$ and $d\prec q$. 
Also, $\chi \cdot \psi$ is even.
Thus, Lemma \ref{afe} implies
\begin{equation*}
    \CM = \sum_{\substack{\chi \Mod{d}\\ \chi \text{ even}}} 2 \sum_{m,n \geq 1} \frac{\chi(m)\psi(m)\overline\chi(n)\overline\psi(n)}{\sqrt{mn}} V\left(\frac{mn}{q}\right).
\end{equation*}
Using $\frac{(1 + \chi(-1))}{2} $ to detect $\chi$ even followed by orthogonality of characters secures \eqref{eq:Mdef}.
\end{proof}

\begin{lemma}\label{diag}
Let $q$ be a positive integer, and let $V$ be as defined in \eqref{V}. Then
\begin{equation*}
    \sum_{(n,q)=1} \frac{1}{n} V\left(\frac{n^2}{q}\right) = \frac{\varphi(q)}{q}\left[\tfrac{1}{2}\log(q) + \gamma_0 + \tfrac{\gamma'}{\gamma}(1/2) + \theta(q)\right] + O(q^{-1/2+\varepsilon}).
\end{equation*}
\end{lemma}

\begin{proof}
This is a standard contour-shifting argument, with a tedious residue calculation.  See for instance \cite[Lemma 4.1]{IS} for details.
\end{proof}

\section{Proof of Theorem \ref{main2}}
  The focus of this section is to prove Theorem \ref{main2}. 

\subsection{Diagonal Term} 
Since $d^2\preceq q$ implies that $d\prec q$, we may apply Lemma \ref{gambit} to give \eqref{eq:Mdef}.
We decompose $\CM$ into three terms:
$    \CM = \CM_{m=n} + \CM_{m>n} + \CM_{m<n}$,
where
\begin{align}
    \CM_{m=n} &:= \varphi(d) \sum_\pm \sum_{\substack{n \equiv \pm n \Mod{d}\\(n,q)=1}} \frac{1}{n} V\left(\frac{n^2}{q}\right),\label{m=n begin}\\
    \CM_{m>n} = \CM_{m>n}(\psi) &:= \varphi(d) \sum_\pm \sum_{\substack{m>n\geq1\\m \equiv \pm n \Mod{d}}} \frac{\psi(m)\overline\psi(n)}{\sqrt{mn}} V\left(\frac{mn}{q}\right),\label{m>n begin}
\end{align}
and where $\CM_{m<n}$ is given by a similar formula to \eqref{m>n begin} but with $1\leq m<n$.

\begin{lemma}\label{m=n end}
Let $(q,2) = 1$ and suppose $d\preceq q$. Then
\begin{equation*}
    \CM_{m=n} = \frac{\varphi(d)}{2}\frac{\varphi(q)}{q}\left[\log(q) + 2\gamma_0 + \tfrac{\Gamma'}{\Gamma}(1/4) - \log \pi + 2 \theta(q)\right] + O(dq^{-1/2+\varepsilon}).
\end{equation*}
\end{lemma}

\begin{proof}
Observe that we cannot simultaneously have $n \equiv -n \Mod{d}$ and $(n,q)=1$ since $d$ is odd.
Applying Lemma \ref{diag} concludes the proof.
\end{proof}

\subsection{Remaining Setup} 
Note that by symmetry
\begin{equation}\label{symmetry}
    \CM_{m<n}(\psi) = \CM_{m>n}\left(\overline\psi\right).
\end{equation}
For this reason we largely focus on the terms with $m>n$.
Applying a dyadic partition of unity to \eqref{m>n begin} results in
$\CM_{m>n} =
\sum_\pm \CM_{m>n}^{\pm}$, with
\begin{equation*}
\CM_{m>n}^{\pm} =     
     \mathop{\sum\sum}_{M,N \text{ dyadic}} \frac{\varphi(d)}{\sqrt{MN}}  \sum_{\substack{m>n\geq1\\m \equiv \pm n \Mod{d}}} \psi(m)\overline\psi(n) W_{M,N}(m,n).
\end{equation*}
Here
\begin{equation}\label{W}
    W_{M,N}(m,n) = \sqrt{\frac{MN}{mn}} V\left(\frac{mn}{q}\right) \eta\left(\frac{m}{M}\right) \eta\left(\frac{n}{N}\right),
\end{equation}
with $W^{(j,k)}(m,n) \ll_{j,k} M^{-j}N^{-k}$ and $\supp(W_{M,N}(m,n)) \subseteq [M,2M]\times[N,2N]$.
By the rapid decay of $V$, we may assume that
\begin{equation}
\label{eq:MNsizes}
MN \ll q^{1+\varepsilon}
\end{equation}

Consider $\CM^{\pm}_{m>n}$.  Say $m= \pm n+dl$ with $l \geq 1$, and define 
\begin{equation}\label{method split}
    \CB^{\pm}_{m>n}(M,N) :=  \sum_{l \geq 1} \sum_{n \geq 1} \psi(\pm n +  dl)\overline\psi(n) W_{M,N}(\pm n + dl,n),
\end{equation}
so that
\begin{equation}\label{bound break}
    \CM^{\pm}_{m>n} = \mathop{\sum\sum}_{M,N \text{ dyadic}} \frac{\varphi(d)}{\sqrt{MN}} \CB^{\pm}_{m>n}(M,N).
\end{equation}

\subsection{Off-diagonal via Poisson in $l$}\label{far}
We will eventually split the dyadic summations of $\CM^{\pm}_{m>n}$ into two ranges depending on whether $M$ and $N$ are nearby or far apart.  In this section we develop a method that will be most useful when $M$ and $N$ are far apart (or, ``unbalanced").
  
  We will apply Poisson summation to the sum over $l$ in \eqref{method split}. First we observe some properties of the function
\begin{equation}\label{W+B}
    W^{\pm}_{n,d}(l) := W_{M,N}(\pm n +  dl,n),
\end{equation} 
namely that $W^{\pm}_{n,d}(l)$ is supported on 
\begin{equation}\label{W_B supp}
    l \in  \Big[\frac{M \mp n}{d},\frac{2M \mp n}{d} \Big]
\end{equation}
and 
\begin{equation}\label{W_B bound}
   \frac{d^j}{dx^j} W_{n,d}^{\pm}(x) \ll_j \left(\frac{M}{d}\right)^{-j}.
\end{equation}
We record some properties of $\widehat W$ for future use.
\begin{proposition}\label{W hat}
Let $A > 0$, and suppose $W(x)$ is a smooth function with support in an interval of length $A$ that satisfies $W^{(j)}(x) \ll_j A^{-j}$.  Then for any $C > 0$ we have
\begin{equation*}
\widehat{W}(x) \ll A (1 + |x| A)^{-C}.
\end{equation*}
\end{proposition}
  These properties follow directly from integration by parts, so the proof is omitted. Applying this to $W_{n,d}^{\pm}$, we have
  \begin{equation}
  \label{eq:Whatbound}
  \widehat{W_{n,d}^{\pm}}(x) \ll_C \frac{M}{d} \Big(1 + \frac{|x| M}{d} \Big)^{-C} \qquad (\forall C > 0).
  \end{equation}

\begin{lemma}\label{method B+}
Suppose $(q, 3)=1$, $d^2\preceq q \preceq d^3$
and that $M > 2^{10} N$.
Then
\begin{equation*}
    \CB^{\pm}_{m>n}(M,N) = \CA^{\pm}_{m>n}(M,N) + O\left(\frac{Nq^{1/2+\varepsilon}}{d}\right),
\end{equation*}
where
\begin{equation}
\label{eq:Aplusdef}
    \CA^{\pm}_{m>n}(M,N) = \left(\frac{-2 a_\psi}{q/d^2}\right) \varepsilon_{q } \frac{d}{\sqrt{q}}\; 
e_q( \overline{2a_\psi}(a_\psi-b_\psi)^2)   
    \sum_{n|b_\psi} \widehat{W^{\pm}_{n,d}} \left(\frac{\mp b_\psi d}{nq}\right).
\end{equation}
\end{lemma}

\begin{proof}
The condition that $M > 2^{10} N$ ensures that $W_{n,d}^{\pm}(l)$ is supported on $l \asymp \frac{M}{d}$ (recall \eqref{W_B supp}).  This is convenient because in particular we can extend the sum over $l \geq 1$ to all $l \in \mz$ without altering the sum.
Applying Poisson summation with respect to $l$ in \eqref{method split} gives 
\begin{equation*}
    \CB^{\pm}_{m>n}(M,N) =  \sum_{n \geq 1}\overline\psi(n) \Big( \frac{d}{q}\sum_{j\in\BZ}\sum_{u \Mod{\frac{q}{d}}} \psi(\pm n+du)\, e_q(dju) \widehat{W^{\pm}_{n,d}}\Big(\frac{dj}{q}\Big)\Big).
\end{equation*}
Replacing $u$ by $\pm nu$, $j$ by $\pm j$, and recalling $\psi$ is even gives
\begin{equation*}
    \CB^{\pm}_{m>n}(M,N) = \frac{d}{q}\; \mathop{\sum_{n\geq1}\sum_{j\in\BZ}}_{(n,q)=1} \widehat{W^{\pm}_{n,d}}\Big(\frac{\pm dj}{q}\Big) 
\underbrace{    
    \sum_{u \Mod{\frac{q}{d}}} \psi(1+du)\, e_q(djnu)
}_{ \CS_{q,d}(\psi, jn) }.
\end{equation*}
Hence, Lemma \ref{sum eval} yields
\begin{equation}
\label{eq:Bmn+formulaPoissoninell}
    \CB^{\pm}_{m>n}(M,N) = \left(\frac{-2 a_\psi}{q/d^2}\right) \varepsilon_{q } \frac{d}{\sqrt{q}}\;\mathop{\sum_{n\geq1}\sum_{j\neq0}}_{jn\equiv -a_\psi\Mod{d}} \widehat{W^{\pm}_{n,d}}\left(\frac{\pm dj}{q}\right) 
e_q(\overline{2a_\psi}(jn+a_\psi)^2),
\end{equation}
where the condition $(n,q)=1$ is now accounted for by $jn\equiv -a_\psi\Mod{d}$, recalling that $(a_{\psi}, q/d) = 1$ and that $d$, $q/d$, and $q$ all share the same prime factors. Since $a_\psi \in \BZ/(q/d)\BZ$ and $d|\frac{q}{d}$, let $a_\psi \equiv b_\psi \Mod{d}$ for $b_\psi \in \BZ/d\BZ$ such that
\begin{equation}\label{b range}
    0<|b_\psi|<\frac{d}{2}.
\end{equation}
The contribution from $jn=-b_\psi$ gives $\CA^{\pm}_{m>n}(M,N)$, while we will use $E.T.$ to denote the remaining terms. Then
\begin{equation}
\label{eq:ETphonehome}
    |E.T.| \leq \frac{d}{\sqrt{q}}\;\mathop{\sum_{n\geq1}\sum_{j\neq0}}_{\substack{jn\equiv -a_\psi\Mod{d}\\jn\neq -b_\psi}} \left|\widehat{W^{\pm}_{n,d}}\left(\frac{\pm dj}{q}\right)\right|.
\end{equation}
Recalling \eqref{eq:Whatbound}
shows this sum may be truncated at $|j| \ll \frac{q^{1+\varepsilon}}{M}$, with a negligible error.
We also have that $n \asymp N$. Therefore, letting $k=jn$, then the non-negligible contribution comes from $|k| \ll \frac{Nq^{1+\varepsilon}}{M}$, and for each $k$, the number of ways to factor $k=jn \neq 0$ is $O(q^{\varepsilon})$. Hence 
\begin{equation}
\label{eq:ETphonehome2}
    E.T. \ll q^{-2022} + q^\varepsilon\Big(\frac{M}{\sqrt{q}}\; \sum_{\substack{0<|k|\ll\frac{Nq^{1+\varepsilon}}{M}\\k\equiv -a_\psi\Mod{d}\\k\neq -b_\psi}} 1\Big) \ll \frac{Nq^{1/2+\varepsilon}}{d}.
\end{equation}
In this estimation we have used that $|k| \geq d/2$ following from $k \equiv - b_{\psi} \pmod{d}$ with $k \neq -b_{\psi}$, and using \eqref{b range}.
\end{proof}

\subsection{Off-diagonal via Poisson in $n$}\label{near}
Returning to \eqref{method split}, we now develop a method designed to treat the complementary range where $M$ and $N$ are nearby. Define the function
\begin{equation*}
    W^{\pm}_{dl}(n) := W_{M,N}(\pm n+dl,n),
\end{equation*}
which is supported on $n \in [N, 2N]$,
and satisfies
\begin{equation}\label{W_A bound}
    \frac{d^j}{dx^j} W_{dl}^{\pm} (x) \ll_j N^{-j},
    \qquad
    \widehat{W_{dl}^{\pm}}(x) \ll N (1+|x|N)^{-C}.
\end{equation}
 
\begin{lemma}\label{A bound}
Suppose $d \prec q$.
We have
\begin{equation*}
    \CB^{\pm}_{m>n}(M,N) \ll q^{\varepsilon}\left( \frac{Mq^{1/2}}{d^{3/2}} + \frac{M^{1/4}N}{q^{1/4}} \right).
\end{equation*}
\end{lemma}

\begin{proof}
We present the details for $\CB^{+}$; the case of $\CB^{-}$ is nearly identical.
Applying Poisson summation to \eqref{method split} in the variable $n$ gives  
\begin{align*}
    \CB^+_{m>n}(M,N) =  
     \frac{1}{q}\;\sum_{l\geq1} \sum_{k\in\BZ} \widehat{W_{dl}^{+}} 
     \left(\frac{k}{q}\right) \sum_{u \Mod{q}} \psi(u+dl) \overline\psi(u) \,  e_q(ku).
\end{align*}
Using \eqref{W_A bound} we may restrict attention to $|k| \ll \frac{q^{1+\varepsilon}}{N}$ with a negligible error.
 Hence
\begin{equation*}
    \CB^+_{m>n}(M,N) \ll q^{-2022} + \frac{N}{q}\;\sum_{1\leq l\ll\frac{M}{d}}\sum_{0\leq|k|\ll\frac{q^{1+\varepsilon}}{N}}  \Big| \sum_{u \Mod{q}} \psi(u+dl)\overline\psi(u)\, e_q(ku) \Big|,
\end{equation*}
since the variable $l$ satisfies $l\ll\frac{M}{d}$ by the support of the test functions. The innermost sum can be recognized as $S(q;\psi,dl,k)$ from \eqref{HB sum}.  
Applying Lemma \ref{HB result} gives
\begin{align*}
    \CB^+_{m>n}(M,N) &\ll MNq^{-1+\varepsilon} + \frac{N}{q}\;\sum_{1\leq l \ll\frac{M}{d}}\sum_{1\leq|k|\ll\frac{q^{1+\varepsilon}}{N}} \Big| \sum_{u \Mod{q}} \psi(u+dl)\overline\psi(u)\, e_q(ku) \Big|\\
    &\ll q^\varepsilon\left(\frac{Mq^{1/2}}{d^{3/2}}+\frac{M^{1/4}N}{q^{1/4}}+\frac{MN}{q}\right).
\end{align*}
The third term above may be dropped since
$\frac{MN}{q} \ll q^{\varepsilon} \frac{M^{1/4}N}{q^{1/4}}$, since $\frac{M}{q} \leq \frac{MN}{q} \ll q^{\varepsilon}$.
\end{proof}

  Out of a convenience which will become evident in Section \ref{assemble}, we would prefer to include $\CA^{\pm}_{m>n}(M,N)$ in Lemma \ref{A bound}, in spite of the fact that it does not naturally manifest using the methods of this section. 
\begin{lemma}\label{method A+}
Suppose $d \prec q$ and $d \leq q^{2/3}$.
We have
\begin{equation*}
    \CB^{\pm}_{m>n}(M,N) = \CA^{\pm}_{m>n}(M,N) + O\left(q^{\varepsilon}\left( \frac{Mq^{1/2}}{d^{3/2}} + \frac{M^{1/4}N}{q^{1/4}} \right)\right).
\end{equation*}
\end{lemma}

\begin{proof}
We treat the $+$ case, since the $-$ case is very similar.
By the triangle inequality and \eqref{eq:Whatbound}, we have
\begin{equation*}
    |\CA^{+}_{m>n}(M,N)| \leq \frac{d}{\sqrt{q}}\;\sum_{n|b_\psi} \left|\widehat{W_{n,d}^{+}}\left(\frac{-b_\psi d}{nq}\right)\right|
\ll \frac{M}{\sqrt{q}} q^{\varepsilon}    
    .
\end{equation*}
Finally, note that $q^{-1/2} \leq d^{-3/2} q^{1/2}$, since $d \leq q^{2/3}$, so this bound on $\mathcal{A}_{m>n}^{+}$ is absorbed by the error term.
\end{proof}

\subsection{Combining $\CM^+_{m>n}$ and $\CM^-_{m>n}$}\label{assemble}
Recall \eqref{bound break}.
Define
\begin{equation}
\label{eq:CWdef}
    \CW = \mathop{\sum\sum}_{M,N \text{ dyadic}} \frac{1}{\sqrt{MN}}\sum_{n|b_\psi} \left(\widehat{W^{+}_{n,d}}\left(\frac{-b_\psi d}{nq}\right) + \widehat{W^{-}_{n,d}}\left(\frac{b_\psi d}{nq}\right)\right).
\end{equation}
As a first step, we have the following.
\begin{lemma}\label{W combo}
We have
\begin{equation}
\label{eq:Wcomboapprox}
    \CW = \frac{\sqrt{q}}{2d}  e_q(-b_{\psi})   \frac{\sigma_0(|b_\psi|)}{\sqrt{|b_\psi|}}
    = \frac{\sqrt{q}}{2d} \frac{\sigma_0(|b_\psi|)}{\sqrt{|b_\psi|}} + O((dq)^{-1/2} q^{\varepsilon})
    .
\end{equation}
\end{lemma}
\begin{proof}
Retracing the definitions, we have
\begin{equation*}
    W^\pm_{n,d}(l)     = 
    \Big(\frac{MN}{(\pm n+dl)n}\Big)^{1/2} V\left(\frac{(\pm n+dl)n}{q}\right) \eta\left(\frac{\pm n+dl}{M}\right) \eta\left(\frac{n}{N}\right).
\end{equation*}
After applying the Fourier transforms and assembling the dyadic partition of unity, we have
\begin{multline*}
    \CW = \sum_{n|b_\psi} \int_{-n/d}^\infty  V\left(\frac{(n+dy)n}{q}\right) e\left(\frac{b_\psi dy}{nq}\right) \frac{dy}{\sqrt{(n+dy)n}} \\+ \sum_{n|b_\psi} \int_{n/d}^\infty  V\left(\frac{(-n+dy)n}{q}\right) e\left(-\frac{b_\psi dy}{nq}\right) \frac{dy}{\sqrt{(-n+dy)n}}.
\end{multline*}
Changing variables results in
\begin{equation*}
    \CW = 2 \frac{\sqrt{q}}{d} 
    e_q(-b_{\psi})
 \sum_{n|b_\psi} \int_0^\infty V(n^2x)  \cos(2 \pi b_{\psi} x) \frac{dx}{\sqrt{x}}.
\end{equation*}
We next wish to
apply the definition of $V(x)$ from \eqref{V} and reverse the order of integration.  This formally gives
\begin{align*}
    \CW 
    &= 2 \frac{\sqrt{q}}{d} e_q(-b_{\psi})\; \sum_{n|b_\psi} \int_{(1/4)} \frac{G(s)}{s} \frac{\gamma(1/2+s)^2}{\gamma(1/2)^2 n^{2s}} \left( \int_0^\infty  x^{-s}  \cos(2 \pi b_{\psi} x) \frac{dx}{\sqrt{x}} \right) \frac{ds}{2 \pi i}.
\end{align*}
This interchange is a little delicate but can be justified in a few ways.  One option is to first truncate
the $x$-integral to a finite interval $[0, X]$ and let $X \rightarrow \infty$ after interchanging the integrals.
Lemma \ref{Mellin} now implies
\begin{equation*}
    \CW =2 \frac{\sqrt{q}}{d} e_q(-b_{\psi})\; \sum_{n|b_\psi} \int_{(1/4)} \frac{G(s)}{s} \frac{\gamma(1/2+s)^2}{\gamma(1/2)^2 n^{2s}}  \frac{\Gamma(1/2-s)}{(2\pi|b_\psi|)^{1/2-s}} \\  
    \cos(\tfrac{\pi}{2}(\tfrac12 - s))
    \frac{ds}{2 \pi i}.
\end{equation*}
Applying the definition of $\gamma(s)$ (found in \eqref{gamma}) and rearranging terms  produces
\begin{equation*}
    \CW =\frac{2\sqrt{q}e_q(-b_{\psi})}{ d \sqrt{2\pi |b_\psi|} \Gamma(1/4)^2} \;  \int_{(1/4)} \frac{G(s)}{s} \Gamma\left(\tfrac{1}{4}+\tfrac{s}{2}\right)^2 \Gamma(1/2-s) \cos\left(\tfrac{\pi}{4}-\tfrac{\pi s}{2}\right)\sum_{n|b_\psi} \left(\frac{2|b_\psi|}{n^2}\right)^s \frac{ds}{2\pi i }.
\end{equation*}
Using standard gamma function identities, one can show easily that
\begin{equation*}
\Gamma\left(\tfrac{1}{4}+\tfrac{s}{2}\right)^2 \Gamma(1/2-s) \cos\left(\tfrac{\pi}{4}-\tfrac{\pi s}{2}\right)
= \pi^{1/2} 2^{-\frac12 - s} \Gamma(\tfrac14 + \tfrac{s}{2})  \Gamma(\tfrac14 - \tfrac{s}{2}).
\end{equation*}
Hence
\begin{equation*}
 \CW =\frac{\sqrt{q}e_q(-b_{\psi})}{ d \sqrt{  |b_\psi|} \Gamma(1/4)^2} \;  \int_{(1/4)} \frac{G(s)}{s} 
 \Gamma(\tfrac14 + \tfrac{s}{2})  \Gamma(\tfrac14 - \tfrac{s}{2}) 
 \sum_{n|b_\psi} \left(\frac{|b_\psi|}{n^2}\right)^s \frac{ds}{2\pi i }.
\end{equation*}
Using the symmetry between divisors, it is apparent that
this integrand is an odd function.  
Therefore, the integral is half the residue at $s=0$, giving the claimed formula for $\CW$.

Finally, we deduce the approximation in \eqref{eq:Wcomboapprox}.  We have $e_q(-b_{\psi}) = 1 + O(q^{-1} |b_{\psi}|)$ by a Taylor expansion, and recalling $|b_{\psi}| \ll d$.  Inserting this into the formula for $\CW$ and using $\sigma_0(|b_{\psi}|) |b_{\psi}|^{1/2} \ll d^{1/2} q^{\varepsilon}$ completes the proof.
\end{proof}

\begin{lemma}\label{m>n end}
With $\CM_{m>n}(\psi)$ as defined in \eqref{m>n begin}, we have
\begin{equation*}
    \CM_{m>n}(\psi) = \CA_{m>n}(\psi) + O\Big(q^\varepsilon\Big(\frac{q^{1/2}}{d^{1/4}}+\frac{d}{q^{1/8}}\Big)\Big),
\end{equation*}
where
\begin{equation*}
    \CA_{m>n}(\psi) =  \left(\frac{-2 a_\psi}{q/d^2}\right) \varepsilon_{q }  \frac{\varphi(d)}{2}  
e_q(\overline{2a_\psi}(a_\psi-b_\psi)^2)    
     \frac{\sigma_0(|b_\psi|)}{\sqrt{|b_\psi|}}.
\end{equation*}
\end{lemma}

\begin{proof}
We begin by splitting the dyadic summations of $\CM_{m>n}$ into two ranges depending on whether $M$ and $N$ are nearby or far apart. The cutoff for these two ranges is $M=d^{1/2}N$. Therefore, starting at \eqref{bound break}, we write
\begin{equation*}
    \CM_{m>n} = \mathop{\sum\sum}_{\substack{M,N \text{ dyadic}\\M\geq d^{1/2}N}} \frac{\varphi(d)}{\sqrt{MN}} \sum_\pm \CB^\pm_{m>n}(M,N) + \mathop{\sum\sum}_{\substack{M,N \text{ dyadic}\\M< d^{1/2}N}} \frac{\varphi(d)}{\sqrt{MN}} \sum_\pm \CB^\pm_{m>n}(M,N).
\end{equation*}
For the first term we apply Lemma \ref{method B+}, and for the second term we apply Lemma \ref{method A+}. Hence
\begin{multline}
\label{eq:Mm>nEstimate}
    \CM_{m>n} = \mathop{\sum\sum}_{\substack{M,N \text{ dyadic}\\M\geq d^{1/2}N}} \frac{\varphi(d)}{\sqrt{MN}} \Big\{\sum_\pm \CA^\pm_{m>n}(M,N) + O\Big(\frac{Nq^{1/2+\varepsilon}}{d}\Big)\Big\} \\+ \mathop{\sum\sum}_{\substack{M,N \text{ dyadic}\\M< d^{1/2}N}} \frac{\varphi(d)}{\sqrt{MN}} \Big\{\sum_\pm\CA^\pm_{m>n}(M,N) + O\Big(q^{\varepsilon}\Big( \frac{Mq^{1/2}}{d^{3/2}} + \frac{M^{1/4}N}{q^{1/4}} \Big)\Big)\Big\}.
\end{multline}
Rearranging and evaluating the main term, 
we obtain
\begin{equation}\label{combin}
    \CM_{m>n} = \CA_{m>n} + \mathop{\sum\sum}_{\substack{M,N \text{ dyadic}\\M\geq d^{1/2}N}} O\Big(\frac{N^{1/2}q^{1/2+\varepsilon}}{M^{1/2}}\Big) + \mathop{\sum\sum}_{\substack{M,N \text{ dyadic}\\M< d^{1/2}N}} O\Big(q^{\varepsilon}\Big( \frac{M^{1/2}q^{1/2}}{N^{1/2}d^{1/2}} + \frac{N^{1/2}d}{M^{1/4}q^{1/4}} \Big)\Big),
\end{equation}
where
\begin{equation*}
    \CA_{m>n} = \mathop{\sum\sum}_{M,N \text{ dyadic}} \frac{\varphi(d)}{\sqrt{MN}} \sum_\pm \CA^\pm_{m>n}(M,N).
\end{equation*}
Using that we may restrict to $N \ll M$ and $MN \ll q^{1+\varepsilon}$, it is easy to see that the error term simplifies to give
\begin{equation*}
    \CM_{m>n}= \CA_{m>n} + O\Big(q^\varepsilon\Big(\frac{q^{1/2}}{d^{1/4}}+\frac{d}{q^{1/8}}\Big)\Big).
\end{equation*}
For the purposes of Theorem \ref{main2}, we have $q \geq d^2$ so that $d q^{-1/8} \leq d^{-1/4} q^{1/2}$, so the latter error term can be dropped.  The displayed error term is consistent with Theorem \ref{main2}.

Now, we turn to $\CA_{m>n}$, which takes the form
\begin{align*}
 \mathop{\sum\sum}_{M,N \text{ dyadic}} \frac{\varphi(d)}{\sqrt{MN}} \left(\frac{-2 a_\psi}{q/d^2}\right) \varepsilon_{q} \frac{d}{\sqrt{q}} 
e_q(\overline{2a_\psi}(a_\psi-b_\psi)^2)
    \sum_{n|b_\psi} \Big( \widehat{W^{+}_{n,d}}\Big(\frac{-b_\psi d}{nq}\Big) +\widehat{W^{-}_{n,d}}\Big(\frac{b_\psi d}{nq}\Big)  \Big)
    \\
    = 
     \varphi(d) \left(\frac{-2 a_\psi}{q/d^2}\right) \varepsilon_{q} \frac{d}{\sqrt{q}} 
e_q(\overline{2a_\psi}(a_\psi-b_\psi)^2) \CW,
\end{align*}
recalling the definition \eqref{eq:CWdef}.
Applying Lemma \ref{W combo} shows that 
\begin{equation*}
 \CA_{m>n} = \CA_{m>n}(\psi) + O(d^{3/2} q^{-1+\varepsilon}).
\end{equation*}
Note $d^{3/2} q^{-1} \leq d q^{-1/8}$, so this error term can be dropped. 
\end{proof}

Next consider $\CM_{m<n}$, for which recall \eqref{symmetry}.  
Lemma \ref{postnikov} implies 
$    a_{\overline\psi} \equiv -a_\psi \Mod{q/d}$, and recalling \eqref{a range}, we have 
$a_{\overline\psi} = -a_\psi$.  
Similarly, $b_{\overline{\psi}} = - b_{\psi}$.  Hence we deduce:
\begin{lemma}\label{m<n end}
We have
\begin{equation*}
    \CM_{m<n}(\psi) = \CA_{m<n}(\psi) + O\Big(q^\varepsilon\Big(\frac{q^{1/2}}{d^{1/4}}+\frac{d}{q^{1/8}}\Big)\Big),
\end{equation*}
where
\begin{equation*}
    \CA_{m<n}(\psi) = 
\left(\frac{ 2 a_\psi}{q/d^2}\right) \varepsilon_{q }  \frac{\varphi(d)}{2}  
e_q(-\overline{2a_\psi}(a_\psi-b_\psi)^2)    
     \frac{\sigma_0(|b_\psi|)}{\sqrt{|b_\psi|}}. 
\end{equation*}
\end{lemma}

\subsection{Combining $\CM_{m>n}$ and $\CM_{m<n}$}
From Lemmas \ref{m>n end} and \ref{m<n end}, we get that
\begin{equation*}
    \CM_{m>n} + \CM_{m<n} = \CA' + O\left(q^\varepsilon\left(\frac{q^{1/2}}{d^{1/4}}+\frac{d}{q^{1/8}}\right)\right),
\end{equation*}
where
\begin{equation*}
    \CA' = \left(\frac{-2 a_\psi}{q/d^2}\right) \varepsilon_{q }  \frac{\varphi(d)}{2}  
\Big[
e_q(\overline{2a_\psi}(a_\psi-b_\psi)^2)    
    + \Big(\frac{-1}{q}\Big) e_q(-\overline{2a_\psi}(a_\psi-b_\psi)^2 ) \Big].
\end{equation*}
Therefore, if $q\equiv 1\Mod{4}$, then
\begin{align*}
    \CA'
    &= \left(\frac{2 a_\psi}{q/d^2}\right) \varphi(d) \cos\Big(2\pi \frac{\overline{2a_\psi}(a_\psi-b_\psi)^2}{q}\Big) \frac{\sigma_0(|b_\psi|)}{\sqrt{|b_\psi|}}.
\end{align*}
This is consistent with \eqref{main2.2} for $q \equiv 1 \pmod{4}$.
If instead $q\equiv 3\Mod{4}$, then
\begin{align*}
    \CA'
    &= \left(\frac{2 a_\psi}{q/d^2}\right) \varphi(d) 
    \sin\Big(2\pi \frac{\overline{2a_\psi}(a_\psi-b_\psi)^2}{q}\Big) \frac{\sigma_0(|b_\psi|)}{\sqrt{|b_\psi|}}.
\end{align*}
This derivation agrees with \eqref{main2.2}. Combining this with Lemma \ref{m=n end} proves Theorem \ref{main2}.

\section{Sketching the Proofs of Remaining Theorems}
The proof of Theorem \ref{main1} is similar to the proof of Theorem \ref{main2}, except in some ways which make it simpler. Likewise, the proof of Theorem \ref{main3} will essentially use a subset of the tools used to prove Theorem \ref{main2}.  In order to avoid excessive repetition, we only give sketches of these proofs.
\subsection{Sketch of the Proof of Theorem \ref{main1}}\label{sketch1}
The structure of the proof of Theorem \ref{main1} is similar to that of Theorem \ref{main2}.  As a substitute for Lemma \ref{method B+}, we have the following.
\begin{lemma}\label{method B+thm1}
Suppose that $d\prec q \preceq d^2$, and that
 $M > 2^{10} N$.
Then
\begin{equation*}
    \CB^{\pm}_{m>n}(M,N) = \CA^{\pm}_{m>n}(M,N) + O\left(Nq^{\varepsilon}\right),
\end{equation*}
where
\begin{equation*}
    \CA^+_{m>n}(M,N) = 
    \sum_{n|b_\psi} \widehat{W^{\pm}_{n,d}}\left(\frac{-b_\psi d}{nq}\right).
\end{equation*}
\end{lemma}
\begin{proof}
We follow through the proof of Lemma \ref{method B+}, and note that the earliest difference will occur at \eqref{eq:Bmn+formulaPoissoninell} when Lemma \ref{sum eval} was used to evaluate $\CS_{q,d}(\psi, jn)$.
 The condition that $q\preceq d^2$ means that we need to use Lemma \ref{sum eval2} in place of Lemma \ref{sum eval}.  
 In practical terms, this means that in place of \eqref{eq:Bmn+formulaPoissoninell} we instead obtain
\begin{equation*}
    \CB^+_{m>n}(M,N) = \mathop{\sum_{n\geq1} \sum_{j\neq 0}}_{jn\equiv -a_\psi\Mod{q/d}} \widehat{W^{\pm}_{n,d}}\left(\frac{dj}{q}\right).
\end{equation*}
In this case, there is no need to introduce $b_{\psi}$, since we have $jn \equiv - a_{\psi} \pmod{q/d}$, and $a_{\psi}$ is inherently defined modulo $q/d$.  The term $\CA$ corresponds to the term $jn = -a_{\psi}$, while the error terms are, similarly to \eqref{eq:ETphonehome} and \eqref{eq:ETphonehome2}, bounded by
\begin{equation*}
 \mathop{\sum_{n\geq1}\sum_{j\neq0}}_{\substack{jn\equiv -a_\psi\Mod{q/d}\\jn\neq -b_\psi}} \left|\widehat{W^{\pm}_{n,d}}\left(\frac{dj}{q}\right)\right|
\ll q^{\varepsilon}  \frac{M}{d} \frac{Nq}{M (q/d)}
\ll q^{\varepsilon} N. \qedhere
\end{equation*}
\end{proof}
Lemma \ref{A bound} holds without changes, since the only assumption there is $d \prec q$.  As in Lemma \ref{method A+}, we can freely insert the term $\mathcal{A}_{m>n}^{\pm}(M,N)$ since it is bounded by $d^{-1} M q^{\varepsilon}$, which is in turn bounded by $d^{-3/2} M q^{1/2+\varepsilon}$.  
The new cutoff in the proof of Lemma \ref{m>n end} is $Mq^{1/2}=Nd^{3/2}$, so this error term is absorbed by the error in Lemma \ref{method B+thm1}.  In place of \eqref{eq:Mm>nEstimate}, we obtain
\begin{multline}
\label{eq:Mm>nEstimatev2}
    \CM_{m>n} = \mathop{\sum\sum}_{\substack{M,N \text{ dyadic}\\M\geq q^{-1/2} d^{3/2} N}} \frac{\varphi(d)}{\sqrt{MN}} \Big\{\sum_\pm \CA^\pm_{m>n}(M,N) + O\Big( N q^{\varepsilon} \Big)\Big\} \\+ \mathop{\sum\sum}_{\substack{M,N \text{ dyadic}\\M< q^{-1/2} d^{3/2} N}} \frac{\varphi(d)}{\sqrt{MN}} \Big\{\sum_\pm\CA^\pm_{m>n}(M,N) + O\Big(q^{\varepsilon}\Big( \frac{Mq^{1/2}}{d^{3/2}} + \frac{M^{1/4}N}{q^{1/4}} \Big)\Big)\Big\}.
\end{multline}
In total, this error term is of size $d^{1/4} q^{1/4+\varepsilon} + d^{} q^{-1/8+\varepsilon}$.  Under the hypotheses of Theorem \ref{main1}, we have $q \leq d^2$ and hence $d^{1/4} q^{1/4} \leq d q^{-1/8}$, so the former term can be discarded.  The error term is then seen to be consistent with the statement of Theorem \ref{main1}.

The assembly of the term $\CA$ is similar to that of $\CA'$, 
though it is simpler since there is no need to introduce $b_{\psi}$,
and leads to
\begin{equation*}
\CA = \frac{\varphi(d)}{d} \sqrt{q} \;
    \frac{\sigma_0(|a_\psi|)}{\sqrt{|a_\psi|}}.
\end{equation*}
This concludes the discussion of the proof of Theorem \ref{main1}.

\subsection{A Sketch of the Proof of Theorem \ref{main3}}\label{sketch2}
Since Theorem \ref{main3} is an upper bound, we can arrange the second moment as follows:
\begin{equation*}
\sum_{\chi \shortmod{d}} |L(1/2, \chi \cdot \psi)|^2
\ll \sum_{\chi \shortmod{d}} \Big| \sum_{M \text{ dyadic}} \sum_m \frac{\chi(m) \psi(m)}{\sqrt{m}} \eta\Big(\frac{m}{M} \Big) V\Big(\frac{m}{\sqrt{q}} \Big) \Big|^2.
\end{equation*}
This uses an approximate functional equation for $L(1/2, \chi \cdot \psi)$ in place of Lemma \ref{afe}.  Applying Cauchy's inequality to take $M$ to the outside of the square, we obtain
\begin{equation*}
\sum_{\chi \shortmod{d}} |L(1/2, \chi \cdot \psi)|^2
\ll q^{\varepsilon} \sum_{M \text{ dyadic}} \sum_{\chi \shortmod{d}} \Big| \sum_m \frac{\chi(m) \psi(m)}{\sqrt{m}} \eta\Big(\frac{m}{M}\Big)   V\Big(\frac{m}{\sqrt{q}} \Big) \Big|^2.
\end{equation*}
The purpose of this trick is to completely avoid the ranges where $M$ and $N$ are far apart.

Squaring this out and applying orthogonality of characters, we obtain a diagonal term of size $\ll d q^{\varepsilon}$.  For the off-diagonal terms we 
essentially arrive at
\begin{equation*}
\sum_{M \text{ dyadic}} \frac{d}{M} |\CB_{m > n}(M,M)| \ll q^{\varepsilon} \left( \frac{q^{1/2}}{d^{1/2}} + \frac{d}{q^{1/8}} \right),
\end{equation*}
using Lemma \ref{A bound} for the final bound.  The second error term can be dropped in comparison to the diagonal term.  In all, we obtain the bound in Theorem \ref{main3}.

\end{document}